\RequirePackage[english]{babel}
\documentclass[12pt]{amsart}

\usepackage[latin1]{inputenc}
\usepackage{amsfonts}\usepackage{amsmath}
\usepackage{amssymb}
\usepackage{enumerate}
\usepackage{geometry}
\usepackage{graphics}
\usepackage[all]{xy}
\usepackage[mathscr]{euscript}
\usepackage{amsthm}
\usepackage{array}
\usepackage{xr}

\newtheorem{thm}{Theorem}[section]
\newtheorem{lm}[thm]{Lemma}
\newtheorem{prp}[thm]{Proposition}
\newtheorem{cor}[thm]{Corollary}

\theoremstyle{definition}
\newtheorem{defi}[thm]{Definition}

\theoremstyle{remark}
\newtheorem{rem}[thm]{Remark}

\newcommand{\dia}{\diamondsuit}

\newcommand{\uhr}{\!\!\upharpoonright\!\!}

\DeclareMathOperator{\NF}{succ}

\DeclareMathOperator{\ot}{ot}
\DeclareMathOperator{\id}{id}
\DeclareMathOperator{\hgt}{ht}

\begin{document}

\title[Optimal Matrices of Partitions and Souslin trees]{Optimal Matrices of Partitions\\
and an
application to
%strongly homogeneous
Souslin trees}
\author{Gido Scharfenberger-Fabian}
\address{Institute of Mathematics and Computer Sciences\\
Ernst-Moritz-Arndt-University\\
Walther-Rathenau-Stra\ss e 47\\
17487 Greifswald, Germany\\
Email: {\tt gsf@uni-greifswald.de}}
% \date{\today}

\begin{abstract}
The basic result of this note is a statement about the existence of families
of partitions of the set of natural numbers with some favourable properties,
the $n$-optimal matrices of partitions.
We use this to improve a decomposition result for strongly homogeneous
Souslin trees.
The latter is in turn applied to separate strong notions of rigidity of
Souslin trees thereby answering a considerable portion of a question of
Fuchs and Hamkins.
\end{abstract}

\subjclass[2000]{Primary 03E05, Secondary 05A18}

\keywords{Souslin trees, homogeneity, rigidity, partition}

\maketitle

\section*{Introduction}

In many models of set theory, Souslin trees offer a variety of different
homogeneity or rigidity properties.
Probably the most prominent homogeneity property for Souslin trees is
\emph{strong homogeneity} (cf. Section \ref{sec:str_hom} for the definition)
which implies that the tree is in a certain sense minimal with respect to its
automorphism group.
On the other hand, a great number of rigidity notions
(i.e. absence of nontrivial automorphisms) for Souslin trees
and an array of implications between most of them are known.
In this paper, which resulted out of a part of the authors PhD thesis
\cite{diss}, we present some interrelations between the class
of strongly homogeneous Souslin trees and that of free trees,
the latter consisting of those Souslin trees which have the strongest known
rigidity properties.

The key result which leads to these correspondences
is a certain method for decomposing a strongly homogeneous Souslin tree
into $n$ free factors
(Theorem \ref{thm:str_hom=free_x_free}, which is a strengthening of a known
though unpublished result).
This decomposition uses an elementary,
but apparently new combinatorial tool,
an $n$-optimal matrix of partitions,
which we introduce in the first section.
As will be seen in Section 2, there are several ways to decompose a strongly
homogeneous Souslin tree into $n$ free trees.
But the construction we give using an $n$-optimal matrix of partitions
enables us to prove strong consequences about the behaviour of the factors
which finally are used in the third section to separate certain notions
of parametrized rigidity for Souslin trees (which are all weakenings of
freeness) in Corollaries \ref{cor:n-free} and \ref{cor:n-free_not_UBP}.

A few words on the structure of the paper and the assumed background
which differs strongly from section to section.
The first section
is about the very elementary notion of $n$-optimal matrices of
partitions and does not assume any prerequisites.
The other two sections treat Souslin trees and their structural properties.
In Section 2 we review strong homogeneity and freeness for Souslin trees
and prove two decomposition theorems for strongly homogeneous
Souslin trees.
The final section collects several rigidity notions for Souslin trees
(most of them taken from \cite{degrigST}) and gives the aforementioned
separation results.
Some definitions and proofs in Section 3 refer to the technique of forcing
which we do not review here.
And even though we give the necessary definitions concerning Souslin trees
at the beginning of Section 2, some acquaintance with this subject will
certainly enhance the reader's understanding
of the constructions in Section 2
(very good references, also on forcing, are,
e.g. \cite{devlin-johnsbraten,kunen,jechneu}).
Anyway, we have made an effort to write a paper that is accessible to an
audience exceeding the circle of experts on Souslin trees.

\section{Optimal matrices of partitions}

The main idea is as follows:
Consider an infinite matrix with $\omega$  rows and $n$ columns
where $n$ is a natural number larger than 1:
$$
\begin{pmatrix}
P_{0,0}&\ldots&P_{0,m}&\ldots&P_{0,n-1}\\
\vdots&\vdots&\vdots&\vdots&\vdots\\
P_{k,0}&\ldots&P_{k,m}&\ldots&P_{k,n-1}\\
\vdots&\vdots&\vdots&\vdots&\vdots\\
\end{pmatrix}
$$
Suppose that the entries of this matrix are partitions of the set $\omega$
of natural numbers.
We want to choose these partitions in a way such that
(i) we get an infinite set whenever we intersect a finite family of subsets
of $\omega$ coming from (distinct) partitions of a single column and
(ii) we get a singleton whenever we intersect $n$ sets belonging
to partitions each coming from different columns.
In the following definition the latter requirement is stated
in a slightly stronger form: we want to obtain a singleton whenever we
intersect $n$ sets not all coming from the same column.
The construction in the proof of Lemma \ref{lm:opt_mtx_prt} actually yields
matrices that satisfy this stronger condition,
and we will use it in the proof of Proposition \ref{prp:decomp}
to derive an additional result.

\begin{defi}
\label{defi:opt_mtx_prt}
For $n\in\omega$,
an $n$\emph{-optimal matrix of partitions}
is a family $(P_{k,m}\mid k\in\omega,\,m<n)$
of infinite partitions $P_{k,m}=(a_i^{k,m}\mid i\in\omega)$ of $\omega$
with the following properties.
\begin{enumerate}[(i)]
\item \emph{Column-wise agreement:}
For all $m<n$ and all $i:k\to\omega$ where $k\in\omega$,
the intersection $\bigcap_{\ell<k}a_{i(\ell)}^{\ell,m}$ is infinite.
\item $n$\emph{-optimality:}
For all maps $(i,k,m):n\to\omega\times\omega\times n$
with $(k(j),m(j))\neq(k(\ell),m(\ell))$ for all $j<\ell<n$
and $m(j)\neq m(\ell)$ for at least one pair $j,\ell<n$
the intersection
$$\bigcap_{j<n} a^{k(j),m(j)}_{i(j)}\quad\text{ is a singleton.}$$
\end{enumerate}
\end{defi}

Note that if in (ii) the domain of $(i,k,m)$ is restricted
to a proper subset of $n$,
i.e., if we intersect over less than $n$ sets,
then the corresponding intersection has to be infinite as well.

\begin{lm}
\label{lm:opt_mtx_prt}
There is an $n$-optimal matrix of partitions for every natural number $n>1$.
\end{lm}
\begin{proof}
To start we fix a bijective enumeration
$h=(h_0,\ldots,h_{n-1}):\omega\to\omega^n$ and define $a_i^{0,m}$
to be the pre-image of $i$ under $h_m$.
Let $P_{0,m}:=\{a_i^{0,m}\mid i\in\omega\}$.

The rest of the construction consists of a three-fold recursion.
The outer loop is indexed with $(k,m)\in\omega\times n$,
and goes row by row, from the left to the right.
One could also say that the progression of the indices
follows the lexicographic order of $\omega\times n$, i.e.,
$m$ grows up to $n-1$ and then drops down to 0 while $k$ increases to $k+1$.
(The first $n$ stages of the outer loop, where $k=0$,
have been included in the recursive anchor in the first line of the proof.)

The inner recursion loops are common $\omega$-recursions.
In each stage of the middle one we define one element $a^{k,m}_i$
of the partition $P_{k,m}$,
and the innermost consists of a choice procedure for the elements
of that set $a^{k,m}_i$.

So assume that the partitions $P_{\ell,m}=\{a_i^{\ell,m}\mid i\in\omega\}$
have already been defined for $(\ell,m)<_\mathrm{lex}(k,n)$
and also the $i$ first sets
$a^{k,m}_0=a_0,\ldots,a^{k,m}_{i-1}=a_{i-1}$ of $P_{k,m}$
have been fixed.
Assume also,
that the family constructed so far
has the properties (i) and (ii) from Definition~\ref{defi:opt_mtx_prt}.
We inductively choose three sequences
$x_\ell,\,y_\ell$ and $z_\ell$ of members of $\omega\setminus\bigcup_{h<i}a_h$
and afterwards set $a_i:=\{x_\ell,y_\ell\mid\ell\in\omega\}$.
The members of the $x$ sequence will make the matrix satisfy
column-wise agreement (requirement (i))
while the $y_\ell$ guarantee that the intersections
for (ii) ($n$-optimality) are non-empty.
The elements $z_\ell$ go back to the stack
and build the pool for the construction of the further members of $P_{k,m}$.

We need the following objects $f,\,b,\,I,\,c,\,\tau$ and $d$
for book-keeping and assume that they
have been fixed at the start of the definition
of the members of the partition $P_{m,k}$, before the construction of $a_0$.

Let $f:\omega\to{^k \omega}$ be onto and $\aleph_0$-to-1
and set for $\ell\in\omega$
$$b(\ell):=\bigcap_{j<k} a^{j,m}_{f(\ell)(j)}.$$
These sets $b(\ell)$ have to be met by $a_i$ infinitely many times.
So we will choose $x_\ell$ from $b(\ell)$.
Let $I$ be the set of subsets $\sigma$ of
$(\omega\times k\times n)\cup(\omega\times(k+1)\times m)$ of
cardinality $n-1$ such that $\mathrm{pr}_{2,3}\uhr\sigma$ is injective and 
$\mathrm{pr}_3\uhr\sigma$ is not equal to to the constant function whose only
value is $m$,
where $\mathrm{pr}_3$ and $\mathrm{pr}_{2,3}$ are the projections to the third
and to the second and third component of a triple respectively.
So if $\sigma\in I$ then the elements of $\sigma$
are indices for $n-1$ members of pairwise distinct partitions and these 
partitions lie in at least two distinct columns of the $n$-optimal matrix.
For $\sigma\in I$ let
$$c(\sigma):=\bigcap_{(j,p,q)\in\sigma} a^{p,q}_j$$
and note that this set is infinite,
because if $P_{r,s}$ is a partition which is not involved in $\sigma$
then $c(\sigma)$ meets every element of $P_{r,s}$ in exactly one natural
number by $n$-optimality.

Also the sets $c(\sigma)$ have to be met by $a_i$.
So we fix a bijective enumeration $\tau$ of $I$ and
choose $y_\ell$ from $c(\tau(\ell))$
unless that set has already been hit by earlier members of $a_i$.

Condition (ii) imposes that each set $c(\sigma)$ be met in only one element.
So once the intersection of $a_i$ and $c(\sigma)$ is non-empty,
that particular set $c(\sigma)$ has to be avoided in later choices
of members of $a_i$. 
We thus define for every natural number $x$ the set
$$d(x):=\bigcup\left\{ c(\sigma)\mid \, \sigma\in I\text{ and }
%(\forall\,(j,p,q)\in\sigma)x\in a^{p,q}_j
x\in c(\sigma)
\right\}$$
and choose $x_\ell$ and $y_\ell$ from outside
$\bigcup_{j<\ell} d(x_j)\cup d(y_j)$.

We now turn to the formal definition of our three sequences
$x_\ell,\,y_\ell$ and $z_\ell$
and argue afterwards, why these choices are always possible.
Set $e=\bigcup_{h<i}a_h$ and let inductively
\begin{eqnarray*}
x_\ell&:=&
\min
 \left(
  b(\ell)
  \setminus
  \left(
   e\cup\bigcup_{j<\ell}(d(x_j)\cup d(y_j)\cup\{z_j\})
  \right)
 \right),\\
y_\ell&:=&
\begin{cases}
x_\ell,&
c(\tau(\ell))
\cap
\bigcup_{j<\ell}\{x_j,y_j,x_{j+1}\}
 \neq
\emptyset\\
\min 
 c(\tau(\ell))
 \setminus 
  (e\cup\{z_0,..,z_{\ell-1}\}),&
\text{otherwise,}
\end{cases}\\
z_\ell&:=&
\min
 \left(
  b(\ell)
  \setminus
  \left(
   e\cup\bigcup_{j<\ell}\{x_j,y_j,z_j,x_{j+1},y_{j+1}\}
  \right)
 \right).
\end{eqnarray*}
We finally show that this construction does not break down, i.e., that the 
sets, from which $x_\ell$ or $z_\ell$ are picked, are non-empty for all
$\ell$.  (A variation shows that the choice $y_\ell$ is always possible, as well.)
The argument splits depending on whether $k$ is less than or at least $n-1$.
So let $k<n-1$ and $\ell\in\omega$, set
$\rho=\rho(\ell)=\{(f(\ell)(j),j,m)\,:\,j<k\}$ and fix $\sigma\in I$ such that $\rho\subset\sigma$,
% We will represent $b(\ell)$ (with respect to $\sigma$) as the disjoint union
% of infinitely many sets of the form $c(\sigma^*)$ and show that for co-fininitely
% many of these components the restrictions for 
say $\sigma=\rho\cup\{(j_1,p_1,q_1),\ldots,(j_r,p_r,q_r)\}$ for $r=n-1-k$.
Now let
$$F:=F(\ell,\sigma):=
\left\{
\{(g_1,p_1,q_1),\ldots,(g_r,p_r,q_r)\}
\,:\,
(g_1,\ldots,g_r)\in\omega^r
\right\}$$
and note that $b(\ell)$ is the disjoint union of the family
$\{c(\rho\cup\tau)\,:\,\tau\in F\}$:
$$b(\ell)=\dot{\bigcup_{\tau\in F}}c(\rho\cup\tau).$$
In order to prove that we always can choose our new element of the $x$-, $y$-
or $z$-sequence, we first state two consequences of $n$-optimality:
\begin{enumerate}
  \item For every $\tau\in F$ and $h<i$ the intersection $c(\rho\cup\tau)\cap a_h$
    is a singleton.
  \item Given any $\sigma'\in I$, for every $\tau\in F$ with
    $\sigma'\ne\rho\cup\tau$ the intersection $c(\rho\cup\tau)\cap c(\sigma') $ has at
    most one element.
  \end{enumerate}
Now note that in order to define $x_\ell$ or $z_\ell$,
we subtract only finitely many sets of the form
$a_h$ or $c(\sigma')$ from $b(\ell)$ .
So in any case, infintely many natural numbers remain in
the pool.

Next let $k\ge n-1$. Then any set of the form $c(\sigma)$ meets $b(\ell)$ in
at most one member by a direct application of $n$-optimality.
Furthermore the $z$-part of the construction yields by induction on
$i$ that $b(\ell)\setminus e=b(\ell)\setminus(\bigcup_{h<i}a_h)$ is an
infinite set.

% implies that the set $b(\ell)$
% (and by induction on $i$ for $e:=\bigcup_{h<i}a_h$ also the difference $(b(\ell)$)
% intersects each 

% There are always candidates for $x_\ell$,
% because the set $b(\ell)\setminus e$ is non-empty
% (by induction on $i$ --- this is the reason for the choice
% of the corresponding $z_\ell$ in earlier steps $a_h$)
% and thus an infinite set.
% For $\sigma\in I$ and $\ell\in\omega$
% the intersection $b(\ell)\cap c(\sigma)$ is either empty or a singleton.
% So we have to avoid only finitely many elements of $b(\ell)\setminus e$
% ($d$ is a finite union of finite sets for all $x_j$ and $y_j$)
% which is of course possible.
% The same argument shows that $z_\ell$ is well-defined.
% Finally for $y_\ell$,
% the set $c(\tau(\ell))$ has met every set $a_h$ for $h<i$ in a unique element.
% So once again we only have to delete
% finitely many elements from an infinite set.

We leave the verification that this construction indeed yields an
$n$-optimal matrix of partitions for the reader.
\end{proof}

There is also a proof of Lemma \ref{lm:opt_mtx_prt}
which uses a more sophisticated
forcing style argument in the construction of the partition $P_{k,m}$,
but since it does not significantly decrease the length of the proof
we stuck to the elementary recursive construction given above.

\section{Strongly homogeneous and free Souslin trees}
\label{sec:str_hom_+_free}
% introductory senctence?

\subsection{Preliminaries on Souslin trees}
\label{sec:prel_ST}

A \emph{tree} is a partial order $(T,<_T)$ where for all $t\in T$
the set of predecessors $\{s\mid s<_T t\}$
is well-ordered by $<_T$.
The elements of a tree are called \emph{nodes}.
For a node $t\in T$ we let $\NF(t)$ be the set of $t$'s immediate successors.
The \emph{height} of the node $t$ in $T$ is the order
type of the set of its predecessors under the ordering of $T$,
$\hgt_T(t):=\ot(\{s\mid s<_T t\},<_T)$.
For an ordinal $\alpha$ we let $T_\alpha$ denote
the set of nodes of $T$ with height $\alpha$.
If $\hgt_T(s)>\alpha$ we let $s\uhr\alpha$ be the unique
predecessor of $s$ in level $\alpha$.

The \emph{height of a tree} $T$, $\hgt T$,
is the minimal ordinal $\alpha$ such that $T_\alpha$ is empty.
An \emph{antichain} is a set of pairwise incomparable nodes of $T$,
so for $\alpha<\hgt T$,
the level $T_\alpha$ is an antichain of $T$.

Nodes, that do not have $<_T$-successors, are called \emph{leaves}, and
$T$ is called \emph{$\kappa$-splitting} or \emph{$\kappa$-branching},
$\kappa$ a cardinal, if all nodes of $T$
have exactly $\kappa$ immediate successors, except for the leaves.

A \emph{branch} is a subset $b$ of $T$ that is linearly ordered by $<_T$ and
closed downwards, i.e. if $s<_T t\in b$ then $s\in b$.
Under the notion of a \emph{normal} tree we subsume the following
four conditions:
\begin{enumerate}[a)]
\item there is a single minimal node called the \emph{root};
\item each node $s$ with $\hgt(s)+1<\hgt T$ has at least two immediate successors;
\item each node has successors in every higher non-empty level;
\item branches of limit length have unique limits (if they are extended in the tree),
i.e., if $s,t$ are nodes of $T$ of limit height
whose sets of predecessors coincide, then $s=t$.
\end{enumerate}
Note that by condition c) leaves can
only appear in the top level of a normal tree.

For a node $t\in T$ we denote by $T(t)$ the set $\{s\in T:t\leq_T s\}$
of nodes above (and including) $t$ which becomes a tree when equipped
with the ordering inherited from $T$.
A tree $T$ is said to be \emph{homogeneous},
if for all pairs $s,t\in T$ of the same height there is
a tree isomorphism (of partial orders) between $T(s)$ and $T(t)$,
the trees of nodes in $T$ above $s$ and $t$ respectively.
For many classes of trees, such as Souslin trees, this is equivalent to 
the condition that for each pair $s,t\in T$ of nodes of the same height there
is an automorphism of $T$ mapping $s$ to $t$.
A tree is \emph{rigid} if it does not admit any non-trivial automorphism.

We will consider two operations on the class of trees: sum and product.
Given trees $(T^i,<_i)$ for $i\in I$, the \emph{tree sum} of this family,
denoted by $\bigoplus_{i\in I}T^i$ is the disjoint union of the sets $T^i$
plus a common root $r\notin \bigcup T^i$.
The tree order $<$ on $\bigoplus T^i$ is
given by the (disjoint) union of the tree orders of summands as well as
the relation $r< t$ for all $t\in \bigcup T^i$.
The height of $\bigoplus T^i$ is given by the ordinal
$1 + \sup\{\hgt T^i: i\in I\}$.

Let now all trees $T^i$ be of height $\mu$.
The \emph{tree product} $\bigotimes_{i\in I}T^i$
over the family $(T^i)_{i\in I}$
is given by the union over the cartesian products of the levels
$T^i_\alpha$:
$$\bigotimes_{i\in I}T^i := \bigcup_{\alpha<\mu}\prod_{i\in I} T^i_\alpha.$$
The product tree order is simply the conjunction of the relations $<_i$.

In order to make a decomposition of a tree into a product feasible
we also introduce the notion of a nice tree equivalence relation.
Let $T$ be a normal and $\aleph_0$-splitting tree and $\equiv$
an equivalence relation on $T$.
Then we say that $\equiv$ is a \emph{nice tree equivalence relation (nice
  t.e.r.)} if $\equiv$ respects levels (i.e., it refines $T\otimes T$),
is compatible with the tree order (i.e., $\hgt(s)=\hgt(r) $ and
$s<t\equiv u > r$ imply $s\equiv r$), the quotient partial order
$T/\!\!\equiv$ of $\equiv$-classes ordered by the inherited partial order, i.e.
$$[s]<_\equiv [t]\quad\iff\quad s< t\,,$$
is a normal and $\aleph_0$-splitting tree and the relation is nice,
by which we mean that for all triples of nodes $s,r,t$ such that
$s\equiv r$ and $t$ is above $s$ there is a node $u\equiv t$, $u$ above
$r$.
Another way to formulate this last property ``niceness'' associates to each
branch $b$ through $T$ a subtree $T^{b}_\equiv:=\bigcup_{s\in
  b}s/\!\!\equiv$ of $T$ and requires that it satisfies point c) in our
definition of normal trees, i.e., every node $t\in T^{b}_\equiv$ has
successors in every higher level of $T^b_\equiv$. 

Now consider the case that a tree $T$ carries nice tree equivalence relations
$\equiv_i$ for $i<n$ such that for every level $\alpha$
and every $n$-sequence of equivalence classes
$c_i\in (T/\!\!\equiv_i)_\alpha$, $i<n$
the intersection $\bigcap c_i$ is a singleton $\{t\}$ with $t\in T_\alpha$.
Then we have a natural isomorphism
between the tree $T$ and the product of its quotient trees,
$\bigotimes_{i<n} T/\!\!\equiv_i$ given by
$t\mapsto (t/\!\!\equiv_i : i< n)$.

We finally come to Souslin trees.
In general,
a \emph{Souslin tree} is a tree $T$ of height $\omega_1$
such that
every family of pairwise incomparable nodes
and also every branch of $T$ is at most countable.
\emph{Unless stated otherwise,
we will only consider normal and $\aleph_0$-splitting Souslin trees.}
In this case the sole absence of uncountable
antichains -- referred to as \emph{the contable chain condition (c.c.c.)} --
already implies that the tree has no cofinal branch.
A main elementary feature of Souslin trees is that their square is
no longer Souslin:
$T\otimes T$ violates the c.c.c. for every Souslin tree $T$.

The naming \emph{Souslin} of these trees stems from their tight connection
to the famous question of
Mikhail Yakovlevich Souslin that was published as \emph{Probl\`{e}me 3)}
on page 223 of the first issue of Fundamenta Mathematicae in 1920
(cf. also \cite{devlin-johnsbraten} or \cite[II.4]{kunen};
the latter transliterates the name as \emph{Suslin}). 
It is well known that the existence of Souslin trees is independent of ZFC,
so whenever we assume that there is some Souslin tree, we make an extra
assumption beyond the realm of standard set theory.

\subsection{Strongly homogeneous and free trees}
\label{sec:str_hom}

We take a closer look at two classes of Souslin trees,
that are widely known among set theorists, although
often under different names, of which we try to state as many as possible.

Strongly homogeneous Souslin trees occur quite often
in set theoretic literature.
In \cite{larson-todorcevic} they are called
coherent Souslin trees and play a central
role in the solution of Katetov's Problem
on the metrizability of certain compact spaces.
Shelah and Zapletal show in \cite[Theorem 4.12]{shelah-zapletal} that
Todorcevic's term for a Souslin tree
in one Cohen real is strongly homogeneous,
Larson gives a direct $\diamondsuit$-construction (\cite[Lemma 1.2]{larson}),
and also Jensen's construc\-tion (under the same hypothesis)
of a 2-splitting, homogeneous tree,
as carried out in \cite[Chapter IV]{devlin-johnsbraten},
is easily seen to yield a strongly homogeneous tree.

\begin{defi}
A Souslin tree $T$ is called \emph{strongly homogeneous} if there is a family
$(\psi_{st}\mid s,t\in T,\,\hgt s=\hgt t)$ which has the following properties:
\begin{enumerate}[1)]
\item $\psi_{st}$ is an isomorphism between
the tree $T(s)$ of nodes above $s$ and the tree $T(t)$ of nodes
above $t$ and $\psi_{ss}$ is the identity.
\item (commutativity) For all nodes $r,s,t$ of the same level of $T$
we have $\psi_{rt}=\psi_{st}\circ\psi_{rs}$.
\item (coherence) For nodes $r,s$ from the same level,
$t$ above $r$ and $u=\psi_{rs}(t)$ we require that $\psi_{tu}$ is
the restriction of $\psi_{rs}$ to the tree $T(t)\subset T(r)$.
\item (transitivity) If $t$ and $u$ are nodes
on the same limit level $T_\alpha$, then there is a level $T_\gamma$
below such that for the corresponding predecessors $r$ of $t$ and $s$ of $u$
we have $\psi_{rs}(t)=u$.
\end{enumerate}
\end{defi}

Some authors call such a family of tree isomorphisms associated to a strongly
homogeneous tree a \emph{coherent family}.

Given any homogeneous tree,
it is easy to define a family on $T$ with the properties 1-3) above.
The crucial property of a coherent family is that of \emph{transitivity},
which means that every limit level is a minimal extension of the initial
segment below with respect to the coherent family on that initial segment. 
Also the automorphism group of a strongly homogeneous Souslin tree $T$
is in a sense minimal, as shown in the following proposition.

\begin{prp}
\label{prp:str_hom_aut}
Every automorphism $\varphi$ of a strongly homogeneous Souslin tree $T$ is
eventually equal to the union of a subset of the coherent family
$(\psi_{st})$, i.e.,
there is a countable ordinal $\alpha$ such that for all
nodes $t$ of height greater than $\alpha$ we have
$$\varphi(t) =
\psi_{t\upharpoonright\alpha,\varphi(t\upharpoonright\alpha)}(t).$$
\end{prp}
This implies that the automorphism group of such a tree has exactly
$2^{\aleph_0}$ elements.
\begin{proof}
It suffices to show that above every node $r\in T$ there is a node
$s$, such that all $t>s$ are mapped by the automorphism
$\varphi$ according to the rule
stated above with $\alpha=\hgt(s)$.

To reach a statement contradicting the transitivity
of the family $(\psi_{st})$,
we assume that there is a node $r\in T$,
such that for each successor $s$ of $r$ there is
a node $t\geq s$, such that $\varphi(t)\neq\psi_{s\varphi(s)}(t)$.
% Without loss of generality we can now assume that $r$ is the root of $T$.
We can inductively choose an increasing sequence of ordinals
$\alpha_n$ such that for all nodes $t\in T_{\alpha_{n+1}}$ we have
$$\varphi(t)\neq
\psi_{t\upharpoonright\alpha_n \varphi(t\upharpoonright\alpha_n)}(t).$$
Let $\alpha$ be the supremum of the $\alpha_n$ and pick any node
$t\in T_\alpha$.
Since $\alpha$ is a limit ordinal and by transitivity of the coherent family
we find an $n\in\omega$ such that
$\varphi(t)=\psi_{t\upharpoonright\alpha_n \varphi(t\upharpoonright\alpha_n)}(t)$
which is of course impossible by the choice of the $\alpha_n$.
\end{proof}

Now we come to \emph{free} trees. Also this property has several different
names, e.g. full
(Jensen, Todor\c{c}evi\'{c}, \cite{GKH,todorcevic_trees_and_orders}) or
'Souslin and all derived trees Souslin'
(Abraham and Shelah, \cite{abraham-shelah_aronszajn_trees,abraham-shelah}).
In the context of \cite{degrigST}
(cf. Section \ref{sec:free} of the present article)
free trees could also be called
'$<\!\!\omega$-fold Souslin off the generic branch'.

\begin{defi}
A normal tree $T$ of height $\omega_1$ is \emph{free}
if for every finite (and non-empty) set of nodes
$s_0,\ldots,s_n$ of $T$ of the same height,
the tree product $\bigotimes_{i=0}^n T(s_i)$ satisfies the c.c.c.
\end{defi}

Free trees are easily seen to be rigid Souslin trees
as the product of two isomorphic relative trees $T(s)$ and $T(t)$
would clearly not be Souslin.
In Section \ref{sec:separating} we will also consider weaker,
parametrized forms of freeness.

\subsection{Decompositions of strongly homogeneous Souslin trees}
\label{sec:dec_free}

We now come to the key result of this paper.
The following theorem is stated in \cite[p.246]{shelah-zapletal}
in the case $n=2$ without proof.
Larson gives the construction of a single free subalgebra
of a strongly homogeneous Souslin algebra in terms of trees
in the proof of Theorem 8.5 in his paper \cite{larson}.
Some ideas in the following proof are borrowed from that construction.

\begin{thm}\label{thm:str_hom=free_x_free}
For every natural number $n>1$ and every
$\aleph_0$-branching, strongly homogeneous Souslin tree $T$
there are free Souslin trees $S_0,\ldots,S_{n-1}$
such that $T\cong \bigotimes_{m<n} S_m$.
\end{thm}

\begin{proof}
Let $T$ be a strongly homogeneous Souslin tree and
denote by $\psi_{s,t}$ the members of the coherent family of $T$.
We inductively (level by level) define $n$ nice t.e.r.s
$\equiv_0,\ldots,\equiv_{n-1}$ with the following properties:
\begin{itemize}
\item $T/\!\!\equiv_m$ is a free Souslin tree for $m<n$.
\item For any sequence $(s_0,\ldots,s_{n-1})\in T_\alpha$
the intersection of the classes $s_m/\!\!\equiv_m$ for $m<n$ is
a singleton: $\bigcap_{m<n} (s_m/\!\!\equiv_m)=\{r\}$ for some $r\in T_\alpha$.
\end{itemize}
The second claim entails the existence of the isomorphism between
$T$ and $\bigotimes_{m<n} T/\!\!\equiv_m$.

Let $(P_{k,m}\,:\, m<n,k\in\omega)$ be an $n$-optimal matrix of partitions,
where we view each $P_{k,m}$ as enumerated by $a_i^{k,m}$, $i\in\omega$.
In order to define t.e.r.s we transfer the whole
matrix of the $P_{k,m}$ to every set $\NF(s)$ for $s\in T$ in a coherent way:
Choose for every $\alpha<\omega_1$
an anchor node $r_\alpha\in T_\alpha$
and a bijection $\sigma_\alpha:\omega\to\NF(r_\alpha)$,
and define for $s\in T_\alpha$ and all indices $i,k,m$ the sets
$$a_i^{k,m}(s):=(\psi_{r_\alpha,s}\circ\sigma_\alpha)''a_i^{k,m}.$$
Then clearly for every $s\in T$, $k\in\omega$ and $m<n$,
the set $P_{k,m}(s):=\{a_i^{k,m}(s)\mid i\in\omega\}$ forms
a partition of $\NF(s)$, and these partitions are linked to each other
by the coherent family in a coherent way, i.e., 
$\psi_{s,t}$ transfers $P_{k,m}(s)$ to $P_{k,m}(t)$.

Fix $m<n$ in order to define $\equiv_m$ on $T$ by recursion on the height.
We will also enumerate the $\equiv_m$-classes of each level in
order type $\omega$, i.e.,
we will fix an onto mapping $h:T\to\omega$,
such that for $s,t\in T_\alpha$ we have
$s\equiv_m t$ if and only if $h(s)=h(t)$.

Choose $P_{0,m}(\text{root})$ as the partition
of the set $T_1=\NF(\text{root})$ and let $\equiv_m$ on level $T_1$
be the equivalence relation with classes
$a_i^{0,m}(\text{root})$ for $i\in\omega$.
Let $h(\text{root})=0$ and choose $h$ on $T_1$ in a way,
such that nodes $s$ and $t$ are $\equiv_m$-equivalent
just in case that their $h$-values coincide.

Next we consider the case where $\alpha$ is a successor ordinal,
$\alpha=\gamma+1$ for some $\gamma<\omega_1$.
Let $s,t\in T_\alpha$ and let $s^-<_T s$ and $t^-<_T t$ be
their direct predecessors on level $\gamma$.
We let  $s\equiv_m t$ if and only if their direct predecessors
are $\equiv_m$-equivalent, $s^-\equiv_m t^-$
(so in particular $h(s^-,m)=h(t^-,m)$),
and if there is $i\in\omega$ such that
$$s\in a_i^{h(s^-),m}(s^-)
\text{ and }t\in a_i^{h(t^-),m}(t^-).$$
In words, the $\equiv_m$-equivalence of the direct predecessors
gives us a natural number $h(s^-)$ and we apply
$P_{h(s^-),m}$ on level $\alpha$ to decide
whether or not $s$ and $t$ are $\equiv_m$-equivalent. 
Extend $h$ to level $T_\alpha$ as described above.

On limit stages $\lambda$ the relation $\equiv_m$ is already determined
by its behaviour below,
and we choose the $h\uhr T_\lambda$ once more in any way such that
$h(s)=h(t)$ is equivalent to
$\equiv_m$-equivalence for nodes $s,t\in T_\lambda$.

Having finished the construction of the relation $\equiv_m$,
we show that it produces a nice t.e.r,
where all properties but niceness follow rather easily from the construction.
So we only deduce niceness.
Letting $s\equiv_m r$ on level $\alpha$ and $t$ above $s$ we claim 
that $\psi_{s,r}(t)\equiv_m t$ and show this
by induction on the height of $t$ above $s$.
For successor stages the claim follows directly
from the construction and the inductive hypothesis,
since the relevant partition $P_{j,m}$
is transferred via $\psi_{s,r}$ by the coherence of the coherent family.
The limit case follows directly from the inductive assumption.
(This property of $\equiv_m$,
that $\equiv_m$-equivalence lifts from $s$ and $r$
to preimages and images under
$\psi_{s,r}$ will be used again in the proof of the Claim below.)

It remains to prove the two properties stated before the construction.
We start with the freeness of $T/\!\!\equiv_m$.
Let $s_0,\ldots,s_{k-1}$ be pairwise non-$m$-equivalent nodes
of the same height $\alpha$ for some natural number $k$.
We write $S_i$ for $(T/\!\!\equiv_m)(s_i/\!\!\equiv_m)$ and try to find
for every antichain $A$ of $\bigotimes_{i<k}S_i$
an antichain $B$ of $T$ of the same cardinality.
This would prove that $T/\!\!\equiv_m$ is a free tree.
We get a hint about where to look for the members of such an antichain
$B$ from the following\\
{\bf Claim.} 
Fix $m<n$ and pairwise non-$m$-equivalent nodes
$s_0,\ldots,s_{k-1}\in T_\alpha$.
For any sequence $(t_0,\ldots,t_{k-1})$ of nodes in $T_\beta$,
with $\alpha<\beta$ and $s_i<t_i$ for $i<k$,
the intersection of the classes
$t_i/\!\!\equiv_m\cap\, T(s_i)$ above the nodes $s_i$,
shifted above $s_0$ by $\psi_{s_i,s_0}$, i.e., the set
$$\bigcap_{i<k} {\psi_{s_i,s_0}}''(t_i/\!\!\equiv_m)$$
is infinite and therefore non-empty.
\begin{proof}[Proof of the Claim.]
By induction on the height $\beta$ of the nodes $t_i$,
starting with $\beta=\alpha+1$.
In this minimal case we have $t_i^-=s_i$.
So the sets ${\psi_{s_i,s_0}}"(t_i/\!\!\equiv_m)$
belong to distinct partitions $P_{h(s_i),m}(s_0)$, $i<k$
and therefore have an infinite intersection
by property (i) of the $n$-optimal matrix.

For the higher successor case $\beta=\gamma+1,\,\alpha<\gamma$,
we simulate this initial situation.
By the inductive hypothesis pick a node
$$r_0\in\bigcup_{i<k}{\psi_{s_i,s_0}}''(t^-_i/\!\!\equiv_m)>s_0,$$
and let $r_i:=\psi_{s_0,s_i}(r_0)>s_i$ for $i<k$.
We then know that $r_i\equiv_m t_i^-$,
so $t_i/\!\!\equiv_m$ has elements above $r_i$.
As a consequence $\bigcup_{i<k}{\psi_{r_i,r_0}}''(t_i/\!\!\equiv_m)$
is infinite by the same argument as above and furthermore
a subset of $\bigcup_{i<k}{\psi_{s_i,s_0}}''(t_i/\!\!\equiv_m)$.

For the case where $\beta$ is a limit ordinal
we choose $\gamma<\beta$ large enough,
such that letting $q_i=t_i\uhr\gamma$ for all $i,j<k$ we have
$\psi_{q_i,q_j}(t_i)=t_j$.
This is possible due to the transitivity of the coherent family.
We also require $\alpha<\gamma$.
The inductive hypothesis gives us a node
$$r_0\in\bigcup_{i<k}{\psi_{s_i,s_0}}''(q_i/\!\!\equiv_m)\subset T_\gamma,$$
which  we copy to $r_i:=\psi_{s_0,s_i}(r_0)$.
By this choice, we also have $r_i\equiv_m q_i$.
We consider $u=\psi_{q_i,r_0}(t_i)$.
By the commutativity of the coherent family this definition of $u$
is independent from the choice of $i<k$.
But then
$$\psi_{s_0,s_i}(u)=\psi_{r_0,r_i}(u)=
\psi_{r_0,r_i}\circ\psi_{q_i,r_0}(t_i)=\psi_{q_i,r_i}(t_i)$$
where the first equation follows from coherence,
the second from the definition of $u$ and the third one from commutativity.
So the property stated above right after the construction of $\equiv_m$
implies that $\psi_{s_i,s_0}(t_i)\equiv_m u$
since $r_i\equiv_m t_i$ for all $i<k$.
This completes the proof of the Claim.
\end{proof}
By virtue of the Claim we can pick for every tuple
$(t_0/\!\!\equiv_m,\ldots,t_{k-1}/\!\!\equiv_m)$
of our antichain $A\subset \bigotimes_{i<k} S_i$
a node $u\in\bigcap_{i<k}{\psi_{s_i,s_0}}''(t_i/\!\!\equiv_m)$ and
collect all these nodes in a set $B$.
Then $B$ is clearly an antichain of $T$ with the same cardinality as $A$.
So we have shown that $T/\!\!\equiv_m$ is indeed a free tree.

Now for the second property.
Let $(s_0,\ldots,s_{n-1})$ be any sequence of nodes of some $T_\alpha$.
We need to show that $\bigcap_{m<n}s_m/\!\!\equiv_m$ has a unique element.
This is done by induction on $\alpha>0$.
Starting with $\alpha=1$ we know that
$(s_m/\!\!\equiv_m)=a_{i_m}^{k_m,m}(\mathrm{root})$
for some $i_m$ and $k_m$.
So property (ii) of our $n$-optimal matrix is all we need here.
For $\alpha=\gamma+1$ we assume that the classes $s_m^-/\!\!\equiv_m$
meet in a single node, say $r\in T_\gamma$.
The set of elements of $s_m/\!\!\equiv_m$ which lie above $r$
is then just $a_{i_m}^{h_m(r),m}(r)$ 
and again property (ii) of the matrix proves the claim.
In the limit case we once more use the transitivity of the coherent family.
So let $\alpha$ be a limit and $\gamma<\alpha$ large enough such that
$\psi_{q_m,q_\ell}(s_m)=s_\ell$ where we abbreviate $s_m\uhr\gamma=q_m$.
For a last time in this proof we use the commutativity of the
coherent family:
Let $r$ be the unique element of the intersection
of the classes $q_m/\!\!\equiv_m$.
Then $t=\psi_{q_m,r}(s_m)$ is well defined and
independent from the choice of $m<n$.
By the lifting property of the equivalence relations stated above,
it follows from $q_m\equiv_m r$, that $s_m\equiv_m t$.
\end{proof}

We now state an algebraic feature which distinguishes our
method of decomposition as just carried out from other decompositions,
namely that partial products of our decomposition are always \emph{rigid},
cf. Remark \ref{rem:decomp}.

\begin{prp}
\label{prp:decomp}
Let $T$ be a strongly homogeneous Souslin tree and assume that
it has been decomposed into a product of $n$ free trees $S_0,\ldots, S_{n-1}$
by the procedure presented in the last proof.
Then the product of less than $n$ pairwise distinct trees from the 
sequence $S_0,\ldots,S_{n-1}$ is a rigid Souslin tree.
\end{prp}
\begin{proof}
It is clear that the product tree is Souslin and that it
is sufficient to show rigidity only for the case of $n-1$ factors, where $n>2$.
So assume that $R:=\bigotimes_{i<n-1}S_i$ admits the automorphism
$\varphi'\ne \id$ and derive a contradiction as follows.

Identifying $T$ and $R\otimes S_{n-1}$ we can lift $\varphi'$
to an automorphism $\varphi = \varphi'\otimes \id$ of $T$.
By Proposition \ref{prp:str_hom_aut} there is a countable ordinal $\alpha$
such that above level $T_\alpha$ the mapping $\varphi$ is given by
a subfamily of the coherent family of $T$.
As $\varphi' \ne \id$ there must be a node
$\bar{s}=(s/\!\!\equiv_0,\ldots,s/\!\!\equiv_{n-1})\in R$ such that 
$\varphi'(\bar{s})\ne \bar{s}$.
We certainly can assume that $\hgt_R(\bar{s})=\alpha$.

Pick an $\ell<n-1$ such that the $\ell$th component $\bar{s}_\ell$
of $\bar{s}$ is not mapped to itself under $\varphi'$.
We fix a representative $s\in \bar{s}_\ell=s/\!\!\equiv_\ell\subset T_\alpha$
and let $q:=\varphi(s)$ in order to get
$\varphi(r)=\psi_{sq}(r)$ for all (immediate) successors of $s$.
This is where we will find a contradiction.

In the construction above we have associated to class $s/\!\!\equiv_\ell$
a natural number $h(s,\ell)$ which defined the index of the partition 
that was used to extend $\equiv_\ell$ to the successors of the members of
the class $s/\!\!\equiv_\ell$.
The same holds for $q=\varphi(s)$, but as $q\not\equiv_\ell s$
we have $h(q,\ell)\ne h(s,\ell)$.
This implies that $\equiv_\ell$ for immediate succesors of $s$ and of $q$
is formed by virtue of different partitions of column $\ell$ of our
partition matrix.
Fix any immediate successor $r$ of $s$.
Since $\varphi'$ is assumed to be a well-defined mapping
$\varphi$ should map the intersection $\bigcap_{i<n-1} r/\!\!\equiv_i$
onto the set $\bigcap_{i<n-1} \varphi(r)/\!\!\equiv_i$. 
In particular, the equation
$${\psi_{sq}}''\left(T(s)\cap\bigcap_{i<n-1} r/\!\!\equiv_i\right)
= \left( T(q) \cap \bigcap_{i<n-1} \psi_{sq}(r)/\!\!\equiv_i\right)$$
should be true.
But this is not the case, because, by $n$-optimality of the matrix
of partitions, the left-hand-side of the above equation intersects
every $\equiv_\ell$-class of the immediate successors of $q$
in exactly one element, while the right-hand-side is a subset
of such an $\equiv_\ell$-class and at the same time an infinite set.
\end{proof}

In the next section we will find corollaries of
Theorem \ref{thm:str_hom=free_x_free} in a similar vein.
I especially mention Lemma \ref{lm:free_to_str_hom} which states that
forcing with such a partial product (as above) turns the complementary
partial product, which was rigid in the ground model,
into a strongly homogeneous Souslin tree in the generic extension.

The following complements Theorem \ref{thm:str_hom=free_x_free}.

\begin{thm}\label{thm:str_hom=str_hom_x_str_hom}
Every $\aleph_0$-branching, strongly homogeneous Souslin tree $T$
is (iso\-morphic to) the tree product
of $n$ strongly homogeneous Souslin trees for any given natural number $n>0$.
\end{thm}
\begin{proof}
This is just a simpler variant of the construction in the proof
of Theorem \ref{thm:str_hom=free_x_free} where
we use only the first row of the matrix of partitions
(or just any bijection between $\omega^n$ and $\omega$).
It is then easy to verify that the coherent family of $T$ descends
to the thus obtained factor trees and renders them strongly homogeneous.
\end{proof}

\begin{rem}
\label{rem:decomp}
\begin{enumerate}[(i)]
\item Though, of course, not every tree product
of two strongly homogeneous Souslin trees is Souslin again
(e.g.\ take $T\otimes T$),
there is a converse to the last theorem:
If $S$ and $T$ are strongly homogeneous Souslin trees and
the tree product $S\otimes T$ satisfies the c.c.c.,
then $S\otimes T$ is a strongly homogeneous Souslin tree as well.
\item We see that there are two essentially distinct ways to decompose
a strongly homogeneous tree into (at least three) free factors.
An application of Theorem \ref{thm:str_hom=str_hom_x_str_hom}
to decompose a given strongly homogeneous Souslin tree $T$ into $\ell$
strongly homogenous factors $S_0,\ldots,S_{\ell-1}$
followed by an $\ell$-fold application of the procedure used in the proof of
Theorem \ref{thm:str_hom=free_x_free} to decompose the tree $S_k$ into $m_k$
free trees $R^k_i$ for $0\leq i < m_k$
never results in the same decomposition
as directly using the proof of Theorem
\ref{thm:str_hom=free_x_free} to decompose $T$ into
$\sum_{k=0}^{\ell-1} m_k$ free factors.
The partial products of the latter decomposition are all rigid by
Proposition \ref{prp:decomp} while
the first also has partial products that are strongly homogeneous.
\end{enumerate}
\end{rem}

\section{Separating high degrees of rigidity}
\label{sec:separating}

In this chapter we review several families of rigidity notions
for Souslin trees, all of them weaker than freeness.
These definitions (except for that of an \emph{$n$-free} Souslin tree)
are all taken from \cite{degrigST}.
Most of these definitions refer to the technique of forcing applied
with a Souslin tree as the forcing partial order.
We do not review forcing here.
But recall,
that forcing with a Souslin tree always assumes the inverse order on the tree
(i.e., trees grow downwards when considered as forcing partial orders,
the root is the maximal element, etc.) and adjoins a cofinal branch.

This section is divided in five short subsections.
The first two introduce the rigidity notions to be considered
and the last three state many and prove some separations between them.
We only give proofs that either are elementary or use the
proof of the Decomposition
Theorem \ref{thm:str_hom=free_x_free}.

\subsection{Parametrized freeness}
\label{sec:free}

Considering the definition of the property \emph{free} for Souslin trees
it is natural to ask whether or not it makes any difference
if the number of the factors in the tree products,
that are required to be Souslin, is bounded.
This leads to the following definition which we rightaway connect
to the definition of \emph{being $n$-fold Souslin off the generic branch}
met in \cite{degrigST}.

\begin{defi}
Let $n$ be a positive natural number.
\begin{enumerate}[a)]
\item We say that a Souslin tree $T$ is \emph{$n$-free} if for every subset
$P$ of size $n$ of some level $T_\alpha$, $\alpha<\omega_1$, the
tree product $\bigotimes_{s\in P} T(s)$ satisfies the c.c.c.
\item A Souslin tree is said to be
$n$\emph{-fold Souslin off the generic branch},
if for any sequence $\vec{b}=(b_0,\ldots,b_{n-1})$
generic for the $n$-fold forcing product of (the inverse partial order
of) $T$ and any node $s\in T\setminus\bigcup_{i\in n}b_i$,
the subtree $T(s)$ of all nodes of $T$ above $s$ is a Souslin tree
in the generic extension $M[\vec{b}]$ (which amounts to requiring that the
adjunction of $\vec{b}$ does not collapse $\omega_1$ and preserves the
c.c.c. of the $T(s)$, $s\notin\bigcup b_i$).
\end{enumerate}
\end{defi}
It is easy to see that a 2-free Souslin tree or a tree which is Souslin
off the generic branch cannot be decomposed as the product of two Souslin
trees. And this common feature is no coincidence.

\begin{prp}\label{prp:free_Sotgb}
For a positive natural number $n$ and a normal Souslin tree $T$
the following statements are equivalent.
\begin{enumerate}[a)]
\item $T$ is $n$-fold Souslin off the generic branch.
\item $T$ is $(n\!+\!1)$-free.
% \item For every antichain $P\subset\B$ of size $n+1$, the free product of the relative algebras corresponding
% to the elements of $P$ is a Souslin algebra, i.e., $\bigoplus_{b\in P}\B\uhr b$ satisfies the countable chain condition.
\end{enumerate}
\end{prp}

\begin{proof}
We start with the implication (b$\to$a).
Assume that $T$ is $n+1$-free and let $\vec{b}=(b_0,\ldots,b_{n-1})$ be
generic for $T^{\otimes n}$, the $n$-fold tree product of $T$ with itself.
Choose $\alpha<\omega_1$ large enough,
such that the nodes $t_i:=b_i(\alpha)$ are pairwise incompatible.
Finally, pick a node $t_n\in T_\alpha$ distinct from all the $b_i(\alpha)$.
By our freeness assumption on $T$,
the product tree $\bigotimes_{i\in{n+1}}T(t_i)$
satisfies the countable chain condition.
But then $M[\vec{b}]\vDash$''$T(t_n)$ is Souslin''
by a standard argument concerning chain conditions in forcing iterations.
Now it is easy to see that $T$ is $n$-fold Souslin off the generic branch.

For the other direction we inductively show that
$T$ is $m$-free for $m\leq n+1$,
assuming that $T$ is $n$-fold Souslin off the generic branch.
The inductive claim is trivial for $m=1$. 
So let $m\geq1$ and let $s_0,\ldots,s_m$ be pairwise distinct nodes
of the same height.
Then for any generic sequence
$\vec{b}=(b_0,\ldots,b_{m-1})$ for $\bigotimes_{i\in m}T(s_i)$
we know that $T(s_m)$ is Souslin in the  generic extension $M[\vec{b}]$.
Finally the two-step iteration $\bigotimes_{i\in m}T(s_i)\ast \check{T}(s_m)$
is isomorphic to $\bigotimes_{i\in m+1}T(s_i)$
and satisfies the countable chain condition.
\end{proof}

This proposition implies that a free tree $T$ is also
\emph{free off the generic branch} in the sense that
in the generic extension obtained by adjoining
a cofinal branch $b$ through $T$, for every node $t\in T\setminus b$,
the tree $T(t)$ is still free.

\subsection{Further types of rigidity}

In Sections 1-4 of \cite{degrigST} different notions of rigidity
for Souslin trees are collected:
(ordinary) rigidity, total rigidity and the unique branch property and their
absolute counterparts,
where absoluteness refers to forcing extensions obtained by
adjoining a generic branch to the Souslin tree under consideration.
In this context also the stronger notion of being
($n$-fold) Souslin off the generic branch is introduced 
which we already considered in the last section.

\begin{defi}
\begin{enumerate}[a)]
\item A Souslin tree $T$ is called \emph{$n$-absolutely rigid}, if $T$ is a rigid tree in the generic extension
obtained by forcing with $T^n$ (or equivalently $T^{\otimes n}$).
\item A Souslin tree is \emph{totally rigid}, if the trees $T(s)$ and $T(t)$ are non-isomorphic for all pairs
of distinct nodes $s$ and $t$ of $T$. It is \emph{$n$-absolutely totally rigid} if it is totally rigid
after forcing with $T^n$.
\item A Souslin tree $T$ has the \emph{unique branch property (UBP)}, if forcing with $T$ adjoins only a single
cofinal branch to $T$. For $n>0$ we say, that $T$ has the $n$-\emph{absolute UBP}, if forcing with
$T^{n+1}$ adjoins exactly $n+1$ cofinal branches to $T$.
\end{enumerate}
\end{defi}

Fuchs and Hamkins prove implications as well as some independencies
between these rigidity notions.
They also give in \cite[Section 4]{degrigST}
a diagram of implications between the
degrees of rigidity that we have approximately
reconstructed here for the convenience of the reader.
% (As it is not involved in our results,
% we have omitted \emph{total rigidity} from the reconstructed diagram.)

\begin{table}[htb]
\begin{center}
\begin{small}
\begin{tabular}{ccccccc}
2-free&$\longleftarrow$&\makebox[2.5cm][c]{3-free}&$\longleftarrow$&\makebox[2.5cm][c]{4-free}&$\longleftarrow$&$\ldots$\\
$\downarrow$&&$\downarrow$&&$\downarrow$&&\\
UBP&$\longleftarrow$&\makebox[2.5cm][c]{absolutely UBP}&$\longleftarrow$&\makebox[2.5cm][c]{2-absolutely UBP}&$\longleftarrow$&$\ldots$\\
$\downarrow$&&$\downarrow$&&$\downarrow$&&\\
totally rigid&$\longleftarrow$&\makebox[3cm][c]{abs. totally rigid}&$\longleftarrow$&
\makebox[3cm][c]{2-abs. totally rigid}&$\longleftarrow$&$\ldots$\\
$\downarrow$&&$\downarrow$&&$\downarrow$&&\\
rigid&$\longleftarrow$&\makebox[2.5cm][c]{absolutely rigid}&$\longleftarrow$&\makebox[2.5cm][c]{2-absolutely rigid}&$\longleftarrow$&$\ldots$\\
&&&&&&
\end{tabular} \end{small}
\end{center}
\caption{Implications between degrees of rigidity for Souslin trees.}
\label{diagram}
\end{table}

Fuchs and Hamkins show that
the part of the diagram to the left and below ``absolutely UBP''
is complete in the sense that there are no further
general implications between these rigidity properties.
They ask whether the rest of the diagram
is complete as well, cf.~\cite[Question 4.1]{degrigST}.
We will show (resp. state) below that there are
neither implications from left to the right
(including downwards diagonals,
cf.~Corollaries \ref{cor:n-free} and \ref{cor:n-free_not_UBP}),
nor from the second to the upper row (Theorem \ref{thm:not_simple_UBP}).

\begin{rem}
\label{rem:diagram}
% \begin{enumerate}[a)]
\item Using a standard $\dia$-construction
scheme for a Souslin tree (e.g., cf. \cite[Section 2]{degrigST})
it is not hard to construct a Souslin tree $T$ with the following two
features:
\begin{itemize}
\item On each level $T_\alpha$ no two distinct nodes
have the same number of immediate successors.
So in particular $T$ is $n$-absolutely totally rigid
for every $n\in \omega$.
\item The substructure $R$ of $T$ obtained by
restricting the supporting set to the nodes on the limit levels of $T$
plus the root, is a homogeneous Souslin tree.
Then in a generic extension obtained by forcing with $T$ there are
many cofinal branches in $R$ and each of them gives rise to a cofinal branch
of $T$, which is thus not a UBP tree.
(In fact, every $\aleph_0$-spiltting Souslin tree can be extended to an
$n$-absolutely totally rigid Souslin tree by inserting new successor levels 
such that every two nodes of the same height have a different number of
immediate successors.)
\end{itemize}
This shows that in Diagram \ref{diagram} there can be no arrows that
point upwards from the two lower rows.
So the only question left open is whether there should be any more arrows
between the two lower rows, but a similar construction as the one alluded to
above should also eliminate those.
% the following item is new but erroneous!
% \item Though semingly weak due to its place in the above diagram, $n$-absolute
%   rigidity has a remarkable consequence: it can preserve $\omega_1$.
% An $\aleph_0$-branching Souslin tree $T$, such that forcing with
% $T^n$ collapses $\omega_1$ cannot 
% be $n$-absolutely rigid. In the extension obtained by adjoining $n$ mutually
% generic branches, $T$ has countable height.
% Provided that $T$ is $\kappa$-splitting for some cardinal $\kappa\le\aleph_0$
% it is homogeneous in the generic extension by a result of Kurepa
% (cf.\cite[p.102]{kurepa}).
% \end{enumerate}
\end{rem}

\subsection{Distinct degrees of freeness}

Our next corollary of (the proof of)
Theorem~\ref{thm:str_hom=free_x_free} gives the separation of
the finite degrees of freeness, i.e.,
it shows that the family of parametrized freeness
conditions is properly increasing in strength.

\begin{cor}\label{cor:n-free}
If there is a strongly homogeneous Souslin tree,
then there is an $n$-free, but not $n+1$-free tree.
\end{cor}
\begin{proof}
Let the strongly homogeneous Souslin tree $T$ be decomposed
as the tree product of $n$ free trees $S_i$ for $i<n$
as carried out in the proof of Theorem~\ref{thm:str_hom=free_x_free}.
We show, that the tree sum of the factors,
$$R:=\bigoplus_{i<n}S_i$$
is an $n$-free but not $n+1$-free Souslin tree.
The Claim used in the proof of Theorem~\ref{thm:str_hom=free_x_free}
remains true in the following variant:\\
{\bf Claim'.} For any pair of sequences $(s_0,\ldots,s_{n-1})$ in $T_\alpha$ and $t_i>s_i$ in $T_\beta$ and
any sequence $m:n\to n$ the intersection
$$\bigcap_{i<n}{\psi_{s_i,s_0}}''t_i/\!\!\equiv_{m(i)}$$
is not empty.

Modulo the obvious changes in the notation,
the proof of the Claim' remains completely the same as before,
exploiting the $n$-optimality of the matrix.
And also with the same argument as above we can construct
an antichain of $T$ from any given antichain of $R$
maintaining the cardinality. So $R$ is $n$-free.

We argue that $R$ is not $n$-fold Souslin off the generic branch.

If $b_i$ is a cofinal branch through $S_i$
then in the generic extension obtained by adjoining
$\vec{b}=(b_0,\ldots,b_{n-1})$, the strongly homogeneous
tree $T$ has a cofinal branch as well,
thus destroying the Souslinity of all subtrees of $R$.
\end{proof}

\subsection{Freeness and absolute rigidity}

In this section we improve upon the result of the last one
by showing that $n$-freeness of a tree does not even imply
$(n\!-\!1)$-absolute rigidity.

\begin{lm}\label{lm:free_to_str_hom}
Let the strongly homogeneous Souslin tree $T$ be
decomposed as the tree product of $n$ free trees $S_i$ for $i<n$
as carried out in the proof of Theorem~\ref{thm:str_hom=free_x_free}.
Let $\{a,b\}$ be a partition of the set $n$ with $a,b\neq\varnothing$
and set $\mathbb{P}:=\bigotimes_{i\in a}S_i$ and $R:=\bigotimes_{i\in b}S_i$.
Then
$$\vDash_{\mathbb{P}}\text{``}\check{R}\text{ is strongly homogeneous.''}$$
\end{lm}
\begin{proof}
We adopt the notation from the statement of the lemma and
% assume without loss of generality that $a=\{0,\ldots,i-1\}$ and
% $b=\{i,\ldots,\n-1\}$ for some $0<i<n$.
% We
argue inside the generic extension obtained by adjoining a generic branch
$c$ to the Souslin tree $\mathbb{P}$.
Then we have the natural isomorphism $\rho:R\cong c\otimes R\subset T$.
Denote the canonical projection $T\to R$ by $\pi$.

We define the tree isomorphisms $\varphi_{{r}{s}}$
(members of the coherent family of $R$ in the generic extension)
for nodes ${r}$ and ${s}$ of $R_\alpha$, $\alpha<\omega_1$
from the members $\psi_{\rho({r})\rho({s})}$ of the coherent family of $T$.
For this, we refer to the maps $h(\cdot,j):T\to\omega$
used in the construction of the t.e.r.s $\equiv_j$ in the
proof of Theorem~\ref{thm:str_hom=free_x_free}.
We collect them and define $h:T\to\omega^{|b|}$ by
simply concatenating the values $h(t,j)$ for $t\in T$ and $j\in b$.

If ${r},\,{s}\in R_\alpha$ and $h(\rho({r}))=h(\rho({s}))$,
then we let
$$\varphi_{{r}{s}}:=
\pi\circ\psi_{\rho({r})\rho({s})}\circ\rho:R({r})\to R({s}).$$
It follows from the fact that the $n$-optimal partition matrices
are transported between the (sets of immediate successors of the)
nodes by the members $\psi_{tu}$ of the coherent family of $T$,
that this definition is sound and indeed yields an isomorphism.

Now let $r,s\in R_\alpha$ with $h(\rho({r}))\neq h(\rho({s}))$.
In order to define $\varphi_{{r}{s}}$ we compose the tree isomorphisms,
that we have already defined for the immediate successors of ${r}$ and ${s}$.
For every direct successor ${u}\in\NF({r})$
there is exactly one ${v}\in\NF({s})$ with $h(\rho({u}))=h(\rho({v}))$.
This follows from the $n$-optimality of the partition matrix.
Let $\varphi_{{r}{s}}({u})$ be just this ${v}$.
If ${x}$ is a non-immediate successor of ${r}$,
then first find the immediate successor ${u}$ of ${r}$
below $x$ and the image $v=\varphi_{rs}(u)$,
and set
$$\varphi_{{r}{s}}({x}):=\varphi_{{u}{v}}({x}).$$

It remains to prove, that the family of tree isomorphisms just defined
is coherent, commutative and transitive.
Commutativity and coherence are inherited from the coherent family of $T$.
(Note that $\varphi_{rs}(x)=y$ implies that $h(\rho(x))=h(\rho(y))$,
so the two cases do not interfere.)
As for transitivity,
let $x,y\in R_\lambda$ for some countable limit ordinal $\lambda$.
Then by the transitivity of the family of the $\psi_{tu}$ for $T$
there are $t<\rho(x)$ and $u<\rho(y)$ with $\psi_{tu}(\rho(x))=\rho(y)$.
But then $t$ and $u$ lie in $b\otimes R$,
so there are $r,s\in R$ such that
$\rho(r)=t$ and $\rho(s)=u$ and thus $\varphi_{rs}(x)=y$.
\end{proof}

So, e.g.\ in the case $n=2$,
forcing with one free tree does not only destroy the freeness of another one,
but even turns the latter into a strongly homogeneous Souslin tree,
i.e., it adjoins many generic automorphisms.

\begin{cor}
\label{cor:n-free_not_UBP}
Let $n>1$.
If there is a strongly homogeneous Souslin tree,
then there is an $n$-free tree
which is not $(n\!-\!1)$-absolutely rigid.
\end{cor}
\begin{proof}
We fix $n>1$ and
use the tree $R$ from the proof of Corollary~\ref{cor:n-free}
obtained from a strongly homogeneous tree $T$
as the tree sum $R=\bigoplus_{i<n}S_i$
of the free factors $S_i$, $i<n$ of $T$.
From Corollary~\ref{cor:n-free} we know, that $R$ is $n$-free.

To show that $R$ is not $(n\!-\!1)$-absolutely rigid we refer
to Lemma~\ref{lm:free_to_str_hom}.
It follows directly from the case that $a=n\setminus\{i\}$ for some $i<n$
that $R$ is not rigid in the generic extension obtained by adjoining
a cofinal branch through the trees $S_j$ for $j<n$ and $j\neq i$.
But this generic extension can also be reached by forcing with
$R^{\otimes n-1}$.
\end{proof}

\subsection{Freeness and the unique branch property}

We start with an easy result deduced from the elementary properties
of finitely free trees for the second column of the diagram.

\begin{prp}
\label{prp:3-free+UBP}
If there is a 3-free Souslin tree,
then there is also a Souslin tree which has the UBP and is not 2-free.
\end{prp}
\begin{proof}
Let $T$ be $3$-free and pick distinct nodes $s,t\in T$ of the same height.
We show, that the Souslin tree $S=T(s)\otimes T(t)$ has the UBP\@.
Let $b\otimes c$ be a generic,
cofinal branch in $S$ (viewing $b$ and $c$ as trees).
By the 2-fold Souslinity off the generic branch of $T$,
every tree of the form $T(r)$ with $r\in T\setminus (b\cup c)$
is Souslin in the generic extension by $b\otimes c$.
On the other hand,
if there was a second cofinal branch through $S$ in the generic extension,
then one of its components would have to pass
through such a node $r\notin b\cup c$, which yields a contradiction.

To prove that $S$ is not $2$-free,
let $u>s$ and $v,w>t$ be of the same height, where $v\ne w$.
Then
$$S(u,v)\otimes S(u,w) \cong T(u)\otimes T(v)\otimes T(u)\otimes T(w)$$
has an uncountable antichain,
because it has the square of the Souslin tree $T(u)$ as a factor.
\end{proof}

This result cannot be improved by simply requiring $T$ to be free, because
by iterating the forcing with a tree product of two factors
$n\!+\!1$ times, we adjoin at least $2^n$ cofinal branches.

We do have the following non-implication result for the $n$-absolute UBP
and $2$-freeness under the stronger assumption of $\dia$.

\begin{thm}
\label{thm:not_simple_UBP}
Assume $\dia$.
Then there is a Souslin tree
which is not 2-free but has the $n$-absolute UBP for all $n\in\omega$.
\end{thm}

The methods of proof for this theorem lie beyond the scope of this paper.
It uses ideas from \cite{degrigST} and \cite{SAE}.
A proof-sketch can be found in \cite[Theorem 1.6.3]{diss}.

\subsection{Further directions}

As a closing remark we mention how Diagram \ref{diagram},
which captures the implications between four families of rigidity notions and
implications between them, could possibly be extended.

\begin{description}
\item[Real rigidity] 
In \cite{abraham-shelah_aronszajn_trees} two Aronszajn trees are called
\emph{really different} if there is no isomorphism between any of their
restrictions to some club set of levels. In this vein, we could call a Souslin
tree 
\emph{really rigid} if all of its restrictions to club sets of levels are
rigid. This property is clearly stronger than ordinary rigidity yet
independent of total rigidity (cf. Remark \ref{rem:diagram}) and is implied
by the unique branching property. 
Also the variant of \emph{real, total
  rigidity} and the $n$-absolute versions of real and of real, total rigidity
could be considered. 

\item[Self-specializing trees] A normal tree $T$ of height $\omega_1$ is called \emph{special} if there is a
countable family $(A_n)_{n\in\omega}$ of antichains of $T$ that covers all of
$T$. As $T$ is uncountable, one of the $A_n$ has to be uncountable as well,
so a special tree $T$ is not Souslin.
On the other hand, every branch of $T$ meets each antichain $A_n$
in at most one node and is therefore countable.

A \emph{self-specializing tree} is a Souslin tree $T$ that specializes itself
by forcing a generic branch $b$ through it, i.e., in the generic extension
obtained by adjoining $b$ to the universe, the tree $T\setminus b$ is special.
Self-specializing trees can be found in models of $\dia$.
They are UBP: a second cofinal branch in $T$ would prevent $T\setminus b$
from being special. But of course they are not Souslin off the generic branch, 
and they can neither be 2-absolutely really rigid nor absolutely UBP,
because forcing with a special tree collapses $\omega_1$,
and in this second generic extension the limit levels of $T$ form an
$\aleph_0$-splitting tree of countable height which must be homogeneous
by a result of Kurepa (cf.\cite[p.102]{kurepa}). 

Now let us call a Souslin tree $T$ \emph{$n$-self-specializing} if it is
$n$-free (i.e. $(n\!-\!1)$-fold Souslin off the generic branch) and
forcing a generic branch $\vec{b}$ through $T^n$ makes $T\setminus\tilde{b}$
special where $\tilde{b}$ is the set of components of the elements of
$\vec{b}$. It is not yet verified but seems quite plausible
that one can construct an $n$-self-specializing tree under $\dia$.
In the implication diagram its place could be between $n$-free
and $(n\!-\!1)$-absolutely UBP, yet it is stronger than both of these
properties. And there would be no horizontal implications, for
an $n$-self-specializing tree is neither $(n\!-\!1)$-self-specializing nor
$(n\!+\!1)$-self-specializing.

% Missing: closing sentence on this idea!

\end{description}
As is clear from the outset, adding these families to Diagram
\ref{diagram} results in a far more complicated directed graph which is in
particular non-planar.
We leave such considerations for future work.

\subsection*{Acknowledgements}
Thanks are due to Piet Rodenburg for pointing out a flaw in the proof of Lemma
\ref{lm:opt_mtx_prt}.

\bibliographystyle{alpha}
\bibliography{souslin,buecher}

\end{document}